\documentclass[11pt]{amsart}
\usepackage{amscd}
\usepackage{amsmath}
\usepackage{amssymb}
\usepackage{latexsym}

\newtheorem{thm}{Theorem}[]
\newtheorem{prop}[thm]{Proposition}
\newtheorem{lem}[thm]{Lemma}
\newtheorem{lem-def}[thm]{Lemma-Definition}

\theoremstyle{definition}

\newtheorem{rmk}{Remark}[]

\newcommand{\frakb}{{\mathfrak b}}

\newcommand{\frakg}{{\mathfrak g}}

\newcommand{\frakl}{{\mathfrak l}}

\newcommand{\frakn}{{\mathfrak n}}

\newcommand{\frakp}{{\mathfrak p}}

\newcommand{\fraks}{{\mathfrak s}}
\newcommand{\frakt}{{\mathfrak t}}

\newcommand{\bbC}{{\mathbb C}}

\newcommand{\bbG}{{\mathbb G}}

\newcommand{\bbP}{{\mathbb P}}

\newcommand{\calF}{{\mathcal F}}

\newcommand{\calO}{{\mathcal O}}

\newcommand{\Ad}{{\mathrm{Ad}}}
\newcommand{\Aut}{{\mathrm{Aut}}}
\newcommand{\Der}{{\mathrm{Der}}}

\newcommand{\Fl}{{\calF\ell}}
\newcommand{\Ga}{{\mathrm{Ga}}}
\newcommand{\Gr}{{\mathrm{Gr}}}

\newcommand{\Lie}{{\mathrm{Lie}}}

\newcommand{\spec}{{\mathrm{Spec}}}

\newcommand{\ord}{{\mathrm{ord}}}

\newcommand{\nc}{\newcommand}
\nc{\on}{\operatorname} \nc{\ch}{\mbox{ch}} \nc{\Z}{{\mathbb Z}}
\nc{\C}{{\mathbb C}} \nc{\pone}{{\mathbb P}^1} \nc{\pa}{\partial}
\nc{\F}{{\mathcal F}} \nc{\arr}{\rightarrow}
\nc{\larr}{\longrightarrow} \nc{\al}{\alpha} \nc{\ri}{\rangle}
\nc{\lef}{\langle} \nc{\W}{{\mathcal W}} \nc{\la}{\lambda}
\nc{\ep}{\epsilon} \nc{\su}{\widehat{{\mathfrak s}{\mathfrak
l}}_2} \nc{\sw}{{\mathfrak s}{\mathfrak l}} \nc{\g}{{\mathfrak g}}
\nc{\h}{{\mathfrak h}} \nc{\n}{{\mathfrak n}}
\nc{\N}{\widehat{\n}} \nc{\G}{\widehat{\g}} \nc{\De}{\Delta}
\nc{\gt}{\widetilde{\g}} \nc{\one}{{\mathbf 1}}
\nc{\z}{{\mathfrak Z}} \nc{\La}{\Lambda} \nc{\wt}{\widetilde}
\nc{\wh}{\widehat} \nc{\cri}{_{\kappa_c}} \nc{\kk}{h^\vee}
\nc{\sun}{\widehat{\sw}_N} \nc{\si}{\sigma} \nc{\el}{\ell}
\nc{\bi}{\bibitem} \nc{\om}{\omega} \nc{\ol}{\overline}
\nc{\ds}{\displaystyle} \nc{\dzz}{\frac{dz}{z}}
\nc{\Res}{\on{Res}} \nc{\mc}{\mathcal} \nc{\Cal}{\mathcal}
\nc{\bb}{{\mathfrak b}} \nc{\ot}{\otimes} \nc{\R}{{\mc R}}
\nc{\yy}{{\mc Y}} \nc{\ga}{\gamma}

\nc{\us}{\underset} \nc{\opl}{\oplus} \nc{\beq}{\begin{equation}}
\nc{\Fq}{{\mathcal F}} \nc{\Mq}{{\mathcal M}} \nc{\Rep}{\on{Rep}}
\nc{\sssec}{\subsubsection} \nc{\ssec}{\subsection}
\nc{\lan}{\langle} \nc{\ran}{\rangle}

\nc{\D}{\mathcal D} \nc{\Vect}{\on{Vect}} \nc{\ghat}{\G}
\nc{\T}{\mc T} \nc{\Tloc}{\T^\g_{\on{loc}}} \nc{\vac}{|0\ran}
\nc{\Wick}{{\mb :}} \nc{\mb}{\mathbf} \nc{\delz}{\partial_z}
\nc{\K}{{\cali K}} \nc{\cali}{\mathcal} \nc{\li}{\mathfrak l}
\nc{\lt}{\widetilde{\li}} \nc{\astar}{a^*} \nc{\cA}{{\mc A}}
\nc{\ka}{\kappa}

\nc{\OO}{{\mc O}} \nc{\AutO}{\on{Aut}\OO} \nc{\DerO}{\on{Der}\OO}
\nc{\DerpO}{\on{Der}_+\OO} \nc{\Au}{{\mc A}ut} \nc{\mf}{\mathfrak}
\nc{\V}{{\mathbb V}} \nc{\hh}{\wh{\h}}

\nc{\pp}{{\mathfrak p}} \nc{\mm}{{\mathfrak m}}
\nc{\rr}{{\mathfrak r}} \nc{\ket}{\rangle} \nc{\zz}{{\mathfrak z}}
\nc{\gr}{\on{gr}} \nc{\Spe}{\on{Spec}} \nc{\rv}{\crho}
\nc{\can}{\on{can}} \nc{\CC}{\on{Op}_G(D))} \nc{\Op}{\on{Op}_G(D)}
\nc{\MOp}{\on{MOp}_G(D)} \nc{\Db}{{\mathbb D}} \nc{\ww}{w}

\nc{\af}{{\mathbb A}^1} \nc{\bs}{\backslash} \nc{\laa}{(\la_i)}
\nc{\zn}{(z_i)}

\nc{\cla}{\check{\la}} \nc{\cmu}{\check{\mu}}
\nc{\crho}{\check{\rho}} \nc{\chal}{\check{\al}}
\nc{\cc}{{\mathfrak c}}

\nc{\M}{{\mathbb M}}

\nc{\ZZ}{{\mc Z}}

\nc{\UU}{{\mathbb U}}

\nc{\Conn}{\on{Conn}(\Omega^{\crho})}
\nc{\Con}{\on{Conn}(\Omega^{-\rho})}
\nc{\Co}{\on{Conn}(\Omega^{\rho})}

\nc{\ppart}{(\!(t)\!)} \nc{\zpart}{(\!(z)\!)}
\nc{\ppzi}{(\!(t-z_i)\!)} \nc{\ppinf}{(\!(t^{-1})\!)}
\nc{\Ind}{\on{Ind}} \nc{\I}{{\mathbb I}} \nc{\ppars}{(\!(s)\!)}
\nc{\QCoh}{\on{QCoh}}

\nc{\aff}{{\on{aff}}}

\begin{document}

\title{Any flat bundle on a punctured disc has an oper
structure}\thanks{Supported by DARPA and AFOSR through the grant
FA9550-07-1-0543}

\author{Edward Frenkel}

\address{Department of
Mathematics, University of California, Berkeley, CA 94720, USA}

\author{Xinwen Zhu}

\date{November 2008}

\begin{abstract}
We prove that any flat $G$-bundle, where $G$ is a complex connected
reductive algebraic group, on the punctured disc admits the structure
of an oper. This result is important in the local geometric Langlands
correspondence proposed in \cite{FG}. Our proof uses certain
deformations of the affine Springer fibers which could be of
independent interest. As a byproduct, we construct representations of
affine Weyl groups on the homology of these deformations generalizing
representations constructed by Lusztig.
\end{abstract}

\maketitle

\section{Introduction}

Let $G$ be a connected reductive algebraic group over $\bbC$,
$\frakg=\mathrm{Lie}(G)$. Let $F=\bbC\ppart$ and $\OO=\bbC[[t]]$.
In this note we prove that every flat $G$-bundle on the formal
punctured disc $D^\times=\spec F$ has an oper structure.  This
proves Conjecture 10.1.1 of \cite{Fr} (see also \cite{Fr2},
Conjecture 1).

By definition, a {\em flat $G$-bundle} (equivalently, de Rham $G$-local
system) on $D^\times$ is a principal $G$-bundle on $D^\times$ with a
connection, which is automatically flat. In concrete terms, the set of
isomorphism classes of flat $G$-bundles is the quotient
\begin{equation}\label{loc}
\mathrm{Loc}_G(D^\times)=\frakg(F)/G(F),
\end{equation}
where $G(F)$ acts on its Lie algebra $\g(F)$ by gauge transformations
as follows:
\begin{equation}\label{Gauge transformation}
\Ga_g(A)=\Ad_g(A)-(\partial_t g)g^{-1}, \ \ \mbox{ for }
A\in\frakg(F), g\in G(F).
\end{equation}
The meaning of the expression 
\begin{equation}\label{dlog} 
d\log(g):=(\partial_tg)g^{-1}
\end{equation} 
as an element in
$\frakg(F)$ is spelled out, e.g., in \cite{Fr} \S 1.2.4.

Let $B\subset G$ be a Borel subgroup. We recall \cite{BD} that a
$G$-{\em oper} is a flat $G$-bundle with a reduction to $B$
satisfying certain conditions. Let us describe the set of
isomorphism classes of $G$-opers on $D^\times$ in concrete terms.
Choose a maximal torus $T\subset B$ and let $\frakt\subset\frakb$
be the corresponding inclusion of Lie algebras. Let $I_{\on{f}}$
be the set of vertices in the finite Dynkin diagram corresponding
to $G$. Let $\alpha_i\in\frakt^*, i\in I_{\on{f}}$ be the set of
simple roots and $X_{-\alpha_i}\in\frakg_{-\alpha_i}$ be a
non-zero root vector corresponding to $-\alpha_i$. (Here, for a
root $\beta\in\frakt^*$, we write $\frakg_\beta$ for the
corresponding root subspace of $\frakg$.) Then the space of
$G$-opers on $D^\times$ is the quotient
\begin{equation}\label{oper form}
\mathrm{Op}_G(D^\times)=\left.\left\{ \left.\sum_{i \in I_{\on{f}}}
\psi_iX_{-\alpha_i} +v \; \right| \; \psi_i\in F^\times,
v\in\frakb(F)\right\}\right/B(F),
\end{equation}
where the action of $B(F)$ is given by (\ref{Gauge
transformation}). Note that if $G$ is semisimple of adjoint type, then
$T(F)$ acts simply transitively on the space of the $\psi_i, i \in
I_{\on{f}}$. Hence the quotient \eqref{oper form} is isomorphic to
\begin{equation}\label{oper form1}
\mathrm{Op}_G(D^\times)=\left.\left\{
\left.\sum_{i \in I_{\on{f}}} X_{-\alpha_i} + v \; \right| \;
v\in\frakb(F)\right\}\right/N(F),
\end{equation}
where $N = [B,B]$ is the unipotent radical of $B$.

There is an obvious forgetful map
\begin{equation}    \label{map}
\mathrm{Op}_G(D^\times)\to\mathrm{Loc}_G(D^\times),
\end{equation}
taking the $B(F)$-gauge equivalent classes to $G(F)$-gauge
equivalent classes.

The main result of this note is

\begin{thm}    \label{main}
The map \eqref{map} is surjective.
\end{thm}

This statement is important in the local geometric Langlands
correspondence developed by D. Gaitsgory and the first author
\cite{FG} (see \cite{Fr} for an exposition). According to
\cite{FG}, to each flat $^L G$-bundle $\sigma$ on $D^\times$ one
should be able to assign a category ${\mc C}_\sigma$ equipped with
an action of the formal loop group $G(F)$ (here $^L G$ is the
Langlands dual group of $G$, which in this paragraph is assumed to
be a simply-connected semisimple complex algebraic group, so that
$^L G$ is of adjoint type). These categories should satisfy some
universality property. In \cite{FG} a candidate for ${\mc C}_\sigma$
was proposed. Namely, let $\chi$ be a pre-image of $\sigma$ in
$\mathrm{Op}_{^L G}(D^\times)$ under the map \eqref{map}, with $G$
replaced by $^L G$ (provided that it exists). Then ${\mc
C}_\sigma$ should be equivalent to the category of modules over
the affine Kac--Moody algebra $\widehat{\g}$ of critical level
with central character determined by $\chi$. This category is
equipped with a natural action of $G(F)$. However, for this
prescription to work for all $\sigma$ it is necessary for the map
\eqref{map} to be surjective.

\begin{rmk}\label{cyclic} A flat $GL_n$-bundle on $D^\times$ is the same as
a rank $n$ vector bundle $\calF$ on $D^\times$ with a connection
$\nabla$. $(\calF,\nabla)$ has an oper structure if and only if there
exists $\phi \in\Gamma(D^\times,\calF)$ such that $\phi,\nabla
\phi,\ldots,\nabla^{n-1}\phi$ generate $\calF$. Such $\phi$ is called
a cyclic vector of $(\calF,\nabla)$.  Therefore, the statement of
Theorem \ref{main} for $G=GL_n$ means that any flat rank $n$ vector
bundle on $D^\times$ has a cyclic vector.  This statement is proved in
\cite{D}, pp. 42--43.
\end{rmk}

\begin{rmk}    \label{kostant}
Let us recall Kostant's theorem \cite{Ko}. Set $f=\sum\limits_{i \in
I_{\on{f}}} X_{-\alpha_i}$. Kostant proved that every {\em regular}
orbit of $\frakg$ intersects with $f+\frakb$. In other words, the map
$$
\left\{ f + v \; | \; v \in {\mathfrak b} \right\}/N \to
  \g^{\on{reg}}/G
$$
is surjective (in fact, an isomorphism), where $\g^{\on{reg}}/G$
denotes the GIT quotient. Therefore, Theorem \ref{main} may be
viewed as an analogue of Kostant's theorem for connections on the
punctured disc (compare with formula \eqref{oper form1}). An
important difference is that a connection can be brought into an
oper form {\em without any regularity assumption}.
\end{rmk}

\begin{rmk} The statement analogous to Theorem \ref{main} for a
smooth projective curve $X$ of genus greater than zero is known to be
false. For instance, if $G$ is of adjoint type, there is a unique (up
to an isomorphism) $G$-bundle on $X$ that can carry an oper structure
(see \cite{BD} \S 3.5). However, it is expected that any flat
$G$-bundle on $X$ has an oper structure with regular singularities at
finitely many points.
\end{rmk}

\noindent{\bf Acknowledgments.} We thank D. Arinkin for suggesting a
simpler proof of Proposition \ref{f.d.}. E.F. thanks Fondation
Sciences Math\'ematiques de Paris for its support and the group
``Algebraic Analysis'' at Universit\'e Paris VI for
hospitality. X.Z. thanks Zhiwei Yun for useful discussions.

\section{Proof of the main theorem in the case when $A_r$ is regular
  nilpotent}

We begin our proof of Theorem \ref{main}. Let $A\in\frakg(F)$.
By taking $r$ small enough, we can always assume that $A$ may be written as
\[A=A_rt^r+A_{r+1}t^{r+1}+\cdots,  \quad\quad r<-1, \quad A_r \mbox{
  is nilpotent. }\]
Here $A_r$ can be zero. First, we have
\begin{lem}\label{reduction} If $A_r$ is regular nilpotent, then
  there exists some $g\in G(\OO)$ such that $\Ga_g(A)$ is an oper.
\end{lem}

\begin{proof} Without loss of generality, we can, and will, assume
  that $A_r = f = \sum_{i \in I_{\on{f}}} X_{-\al_i}$. Let $e \in {\mathfrak
  b}$ be the unique element such that $\{e,2\check{\rho},f\}$ is a
  principal $\fraks\frakl_2$-triple. Let $\frakg^e$ be the centralizer
  of $e$ in $\frakg$. Let $G^{(1)}(\OO)$ be the first congruence
  subgroup of $G(\OO)$, i.e. the kernel of the evaluation map
  $G(\OO)\to G$. We prove that there is some $g\in G^{(1)}(\OO)$ such
  that $\Ga_g(A)\in ft^r+\frakg^e(F)$, which is in the oper form.

According to representation theory of $\fraks\frakl_2$, we have
$\frakg=\frakg^e+\mathrm{ad}f(\frakg)$. Therefore, there exists
$X_1\in\frakg$ such that $A_{r+1}+[X_1,f]\in\frakg^e$. Let
$g_1=\exp(tX_1)$. Since $r<-1$,
\[\Ga_{g_1}(A)=ft^r+(A_{r+1}+[X_1,f])t^{r+1}+\wt{A}_{r+2}t^{r+2}+
\cdots.\]
Next, there exists some $X_2\in\frakg$ such that
$\wt{A}_{r+2}+[X_2,f]\in\frakg^e$. Let $g_2=\exp(t^2X_2)$. Again,
since $r<-1$,
\[\Ga_{g_2}(\Ga_{g_1}(A))=ft^r+(A_{r+1}+[X_1,f])t^{r+1}+
(\wt{A}_{r+2}+[X_2,f])t^{r+2}+\wt{A}_{r+3} t^{r+3}+\cdots\]
By induction, we can find $g_1,\ldots,g_{k-1}$ such that the
coefficients of $t^{r+1},\ldots,t^{r+k-1}$ of
$\Ga_{g_{k-1}}\cdots\Ga_{g_1}(A)$ are in $\frakg^e$. Let
$\wt{A}_{r+k}$ be the coefficient of $t^{r+k}$ in
$\Ga_{g_{k-1}}\cdots\Ga_{g_1}(A)$. Let $X_k\in\frakg$ such that
$\wt{A}_{r+k}+[X_k,f]\in\frakg^e$ and let $g_k=\exp(t^kX_k)$. Then
the coefficient of $t^{r+k}$ in $\Ga_{g_k}\cdots\Ga_{g_1}(A)$
belongs to $\frakg^e$, while the coefficients of
$t^r,\ldots,t^{r+k-1}$ remain unchanged. Let $g=\cdots g_k\cdots
g_2g_1$. This is a well-defined element in $G^{(1)}(\OO)$ which
satisfies the requirement of the lemma.
\end{proof}

\begin{rmk} Let $A_r$ be an arbitrary regular element of $\g$. By
  Kostant's theorem (see Remark \ref{kostant}), we can assume, without
  loss of generality, that $A_r = f + v, v \in {\mathfrak b}$. By a
  slight modification of the above argument, we can then also prove
  that there exists $g\in G^{(1)}(\OO)$ such that $\Ga_g(A)$ is an
  oper. Thus, we obtain a simple proof of the statement of Theorem
  \ref{main} in the case when the leading term $A_r$ is regular. The
  real challenge is to prove that it holds even without this
  assumption.\qed
\end{rmk}

By the previous lemma, in order to prove Theorem \ref{main} it
suffices to prove that there exists $g\in G(F)$ such that
$B=\Ga_g(A)=B_rt^r+B_{r+1}t^{r+1}+\cdots$, with $B_r$ regular
nilpotent. Recall that we are under the assumption $r<-1$. The rest of
this paper is devoted to proving this fact.

\section{Deformed affine Springer fibers}

If $ A=A_rt^r+A_{r+1}t^{r+1}+\cdots$ with $A_r\neq 0$, we call $r$
the {\em order} of $A$, and sometimes denote it by $\ord(A)$. Let
\[M_A=\{g\in G(F) \; | \; \ord(\Ga_{g^{-1}}(A))\geq\ord(A)=r\}.\]
This is a subset of elements $g$ of $G(F)$ which is the set of
solutions of certain algebraic equations on the coefficients of
$g$. Hence it is clear that it is the set of points of an
ind-subscheme of $G(F)$. It is clearly invariant under the right
multiplication by elements of the subgroup $G(\OO)$. Therefore the
quotient
\[Y_A:=M_A/G(\OO)\]
is a well-defined closed ind-subscheme of the affine Grassmannian
$\Gr=G(F)/G(\OO)$. We call it the {\em deformed affine Springer fiber}
associated to $A$.

\medskip

Let us explain this terminology. Set $\wt{A}=t^{-r}A\in\frakg(\OO)$.
For $\la \in \C$, let
\[Y_{\wt{A},\lambda}=\{g\in G(F),
\Ad_{g^{-1}}(\wt{A})-\lambda
t^{-r}d\log(g^{-1})\in\frakg(\OO)\}/G(\OO),\] 
where $d\log(g)$ is defined as in \eqref{dlog}. Then
$Y_{\wt{A},1}=Y_A$, and $Y_{\wt{A},0}$ is the affine Springer
fiber of $\wt{A}$ defined by Kazhdan and Lusztig in \cite{KL} (see
also \cite{GKM}). Let us first show that

\begin{lem}\label{leading term}
For any $gG(\OO)\in Y_A$,
\[(\Ad_{g^{-1}}(\wt{A})-t^{-r}d\log(g^{-1}) \mod
t)\in\frakg(\OO)/t\frakg(\OO)=\frakg\] is nilpotent.
\end{lem}

\begin{proof} Let $T$ be the maximal torus of $G$ whose Lie algebra
is $\frakt$. Let $X_*(T)$ be the coweight lattice of $T$ and
$X_*(T)_+$ be semi-group of dominant coweights. Each
$\check{\lambda}\in X_*(T)$ defines a point
$t^{\check{\lambda}}\in T(F)\subset G(F)$. We have the Birkhoff
decomposition
$$G(F)=\bigsqcup\limits_{\check{\lambda}\in
X_*(T)_+}G(\OO)t^{\check{\lambda}} G(\OO).$$ Let $g\in G(F)$ be as
in the lemma. We can write it as $g=g_1t^{\check{\lambda}}g_2$ for
$g_1,g_2\in G(\OO)$ and a dominant coweight $\check{\lambda}$.
Then we have
\[\Ga_{g_2^{-1}}(B)=\Ga_{t^{\check{\lambda}}}\Ga_{g_1}(A).\]
It is clear that
\[\begin{split}
&C=\Ga_{g_1}(A)=C_rt^r+C_{r+1}t^{r+1}+\cdots,\\
&D=\Ga_{g_2^{-1}}(B)=D_rt^r+D_{r+1}t^{r+1}+\cdots,
\end{split}\]
with $C_r$ nilpotent. We need to show that $D_r$ is nilpotent.

Let $\frakg=\sum\limits_i\frakg_i$ be the weight decomposition of
$\frakg$ with respect to $\check{\lambda}$. Then $\frakg_{\geq
0}:=\sum_{i\geq 0}\frakg_i$ is a parabolic subalgebra of $\frakg$,
and $\frakg_{>0}:=\sum_{i>0}\frakg_i$ is its nil-radical and
$\frakg_0$ is a Levi subalgebra. Similarly, one has $\frakg_{\leq
0}$ and $\frakg_{<0}$. We observe that
$X=X_0+X_{>0}\in\frakg_0+\frakg_{>0}$ (resp.
$X=X_0+X_{<0}\in\frakg_0+\frakg_{<0}$) is nilpotent in $\frakg$ if
and only if $X_0$ is nilpotent in $\frakg_0$.

Now since $D=\Ga_{t^{\check{\lambda}}}(C)$, we know that
$C_r\in\frakg_{\geq 0}$ and $D_r\in\frakg_{\leq 0}$. Furthermore,
if we decompose $C_r=C'+C''$ with
$C'\in\frakg_0,C''\in\frakg_{>0}$ and $D=D'+D''$ with
$D'\in\frakg_0, D''\in\frakg_{<0}$, then $C'=D'$. Since $C_r$ is
nilpotent, $C'=D'$ is nilpotent. Therefore $D_r$ is nilpotent.
\end{proof}

By Lemma \ref{reduction}, Theorem \ref{main} holds if there is a
point $gG(\OO)\in Y_A$ such that the above element is regular
nilpotent. Such a point $gG(\OO)$ is called a {\em regular point}
of $Y_A$ (cf. \cite{GKM2}). Therefore, the main theorem follows
from

\begin{thm}\label{regular point} If $A_r$ is nilpotent (equivalently,
$\wt{A} \mod t$ is nilpotent) and $r\leq -2$, then $Y_A$ has a regular
point.
\end{thm}

Let us interpret this theorem more geometrically. Let $I$ be the Iwahori
subgroup of $G(F)$, i.e. the pre-image of $B\subset G$ under the
evaluation map $G(\OO)\to G$, and $\Fl=G(F)/I$ be the affine
flag variety. Without loss of generality, we can assume that
$\wt{A}\mod t\in\frakn$, where $\frakn$ is the nil-radical of
$\frakb$. Let
\[X_A=\{g\in G(F) \; | \; \Ad_{g^{-1}}(\wt{A})-t^{-r}d\log(g^{-1})\in
\Lie I\}/I\in\Fl.\] There is a natural projection $\pi:X_A\to
Y_A$. The following lemma is clear.

\begin{lem}\label{fiber} A point $p=gG(\OO)\in Y_A$ is a regular point if and
only if $\pi^{-1}(p)$ consists of a single point.
\end{lem}

Observe that to prove Theorem \ref{regular point}, it is enough to prove that $Y_A\cap\Gr^0$ has a regular
point, where $\Gr^0$ is the neutral component of $\Gr$.  Let
$\tilde{G}$ be the simply-connected cover of the derived group of $G$,
and write $A=A_0+A_1$, where $A_0\in\Lie(Z(G)^0)(F)$ ($Z(G)^0$ being
the neutral component of the center $Z(G)$ of $G$) and
$A_1\in\Lie(\tilde{G})(F)$. Then $Y_A\cap\Gr^0$ is (topologically)
isomorphic to $Y_{A_1}$, and $X_A\cap\Fl^0$ is (topologically) isomorphic to $X_{A_1}$. In addition, the projection $\Fl^0\to\Gr^0$ is
(topologically) isomorphic to the map
$\Fl_{\tilde{G}}\to\Gr_{\tilde{G}}$, where $\Fl_{\tilde{G}}$
(resp. $\Gr_{\tilde{G}}$) denotes the affine flag variety (resp.
affine Grassmannian) of $\tilde{G}$. Therefore, the map $X_A\cap\Fl^0\to Y_A\cap\Gr^0$ is (topologically) isomorphic to $X_{A_1}\to Y_{A_1}$. Then according to Lemma \ref{fiber}, it is sufficient to
prove Theorem \ref{regular point} for connected simply-connected
semisimple algebraic groups. Hence, from now on, we will assume that
$G$ is a connected simply-connected semisimple algebraic group.

An analogous statement for non-deformed affine Springer fibers has
been proved in \cite{KL} \S 4. By imitating their proof, we find that
it is sufficient to prove two propositions. The first one is the
following:

\begin{prop}\label{f.d.} $Y_A$ is finite-dimensional.
\end{prop}

Next, we formulate the second proposition.  Recall that the
affine Weyl group $W_{\aff}$ of $G(F)$ acts on $H_*(\Fl)$ (cf.
\cite{Ka} \S 2.7), where $H_*(\cdot)$ stands for the Borel--Moore
homology.

\begin{prop}\label{action of affine Weyl group} Assume that $A_r$ is
  nilpotent (possibly, equal to zero) and $r\leq -2$. Then the image
  of $H_*(X_A)\to H_*(\Fl)$ is invariant under the action of the
  affine Weyl group $W_{\aff}$.
\end{prop}

For the sake of completeness, let us repeat the argument from
\cite{KL} \S 4 that shows how the above two propositions imply Theorem
\ref{regular point}.

Let $d=\dim X_A$ (it is finite by Proposition \ref{f.d.}). Let $X$
be an irreducible component of $X_A$ of dimension $d$. Denote by
$[X]\in H_{2d}(\Fl)$ the homology class represented by $X$. Then
$[X]\neq 0$ by \emph{loc. cit.} \S 4, Lemma 6. Let $V_A$ be the
image of $H_{2d}(X_A)\to H_{2d}(\Fl)$. Then $V_A$ is generated by
these $[X]$. By Proposition \ref{action of affine Weyl group},
$V_A$ is a subrepresentation of the representation of $W_{\aff}$
on $H_{2d}(\Fl)$. By \emph{loc. cit.} \S 4, Lemma 8, $V_A$ has a
non-zero invariant vector under the action of the finite Weyl
group $W_{\on{f}}\subset W_{\aff}$. For $i\in I_{\on{f}}$, let
$P_i$ be the parahoric subgroup of $G(F)$ with Lie algebra $\Lie
P_i=\Lie I+\frakg_{-\al_i}$. Let $\Fl_i$ be the partial affine
flag variety of parahoric subgroups of $G(F)$ which are conjugate
to $P_i$, and $\pi_i:\Fl\to\Fl_i$ be the projection. Assume that
$Y_A$ does not contain a regular point. Then for any $p=gI\in
X_A$, $(\Ad_{g^{-1}}(\wt{A})-t^{-r}d\log(g^{-1}) \mod t)$ is an
element of $\frakn$ (by Lemma \ref{leading term}) which is not
regular, and therefore is contained in the nil-radical of some
parabolic subalgebra $\frakp_i=\frakb+\frakg_{-\alpha_i}, \ i\in
I_{\on{f}}$. In this case, for any $g'\in gP_i$,
$(\Ad_{g'^{-1}}(\wt{A})-t^{-r}d\log(g'^{-1}) \mod t)$ is also
contained in $\frakn$ (in fact, in the nil-radical of $\frakp_i$)
and therefore $\pi_i^{-1}(\pi_i(p))\subset X_A$. For each
$d$-dimensional irreducible component $X\subset X_A$, let $X_i,
i\in I_{\on{f}}$ be the closed subset of points $p$ on $X$ such
that $\pi_i^{-1}(\pi_i(p))\subset X_A$. Then $X=\cup_{i\in
I_{\on{f}}}X_i$. Since $X$ is irreducible, $X=X_i$ for some $i$,
i.e., there exists some $i\in I_{\on{f}}$ such that
$X=\pi_i^{-1}(\pi_i(X))$. Let $T_{s_i}$ be the corresponding
simple reflection in $W_{\on{f}}$, which acts on $H_{2d}(\Fl)$.
Then $(\mathrm{Id}+T_{s_i})[X]=0$. Now let $T=\sum_{w\in
W_{\on{f}}}T_w$. Since for any $i\in I_{\on{f}}$,
$T=Q_i(\mathrm{Id}+T_{s_i})$, we find that $T[X]=0$ for any
$d$-dimensional irreducible component $X\subset X_A$. Therefore,
$TV_A=0$, which contradicts the fact that $V_A$ has a non-zero
invariant vector under the action of $W_{\on{f}}$.

\medskip

In the remaining part of this note we prove Propositions \ref{f.d.}
and \ref{action of affine Weyl group} about the deformed affine
Springer fibers. We also discuss the action of the affine Weyl group
on $H_*(X_A)$.

\section{Proof of Proposition \ref{f.d.}}

We begin with the proof of Proposition \ref{f.d.}. Recall the
definition of $M_A$ from the beginning of last section. For any
$g\in M_A$, we consider the following $\bbC$-vector space
\[T_{g}=\{X\in\frakg(F)|\partial_tX+[B,X]\in t^r\frakg(\OO)\}/\frakg(\OO)\]
where $B=\Ga_{g^{-1}}(A)$. Observe that if $g'=gg_1$ with $g_1\in
G(\OO)$, then $B'=\Ga_{g'^{-1}}(A)=\Ga_{g_1^{-1}}(B)$, and there
is a canonical isomorphism $\gamma_{g_1}:T_g\cong T_{g'}$ given by
$X\mapsto \Ad_{g_1^{-1}}(X)$. Therefore, $T_g$ is canonically attached
to every $gG(\OO)\in Y_A$. From the definition of $Y_A$, it is
clear that $T_g$ is canonically isomorphic to the tangent space of
$Y_A$ at $gG(\OO)$. We claim that the dimension of this
$\bbC$-vector space is $\leq (-r)\dim\frakg$. This proves that the
dimension of $Y_A$ is $\leq (-r)\dim\frakg$.

We regard $\frakg(F)$ as a vector space over $F$, with a connection
$\nabla_t=\partial_t+\text{ad}(B)$. Then
$T_g=\nabla^{-1}(t^r\frakg(\OO))/\frakg(\OO)$. Now the claim is a
direct consequence of the following lemma, whose proof was
suggested to us by D. Arinkin.

Let $(V,\nabla)$ be a finite-dimensional vector space over $F$
with a connection. By an $\OO$-\emph{lattice} in $V$ we understand
a finite generated $\calO$-submodule $L$ of $V$ such that the natural map
$L\otimes_\calO F\to V$ is an isomorphism. By a \emph{lattice} in
$V$ we understand a $\bbC$-subspace in $V$ that is commensurable
with an $\OO$-lattice.

\begin{lem} For any lattice
$L\subset V$, $\nabla^{-1}(L)$ is also a lattice of $V$, and the
relative dimension of $\nabla^{-1}(L)$ to $L$ is
\[[\nabla^{-1}(L):L]:=\dim\frac{\nabla^{-1}(L)}{\nabla^{-1}(L)\cap L}
-\dim\frac{L}{\nabla^{-1}(L)\cap L}\leq 0.\]
\end{lem}

\begin{rmk} This lemma is an easy consequence of Deligne's theory of
``good lattices'' for connections (cf. \cite{D} pp.110-112), as we
learned from D. Arinkin. However, to prove the existence of ``good
lattices'', Deligne used the existence of the cyclic vector for
$(V,\nabla)$ (cf. Remark \ref{cyclic}).  Therefore we prefer to avoid
using these results in the proof of our theorem.
\end{rmk}

\begin{proof} We first recall that the connection $(V,\nabla)$ is
said to be in the \emph{canonical form} (with respect to some
$F$-basis $\mathbf{e}$ of $V$), if it looks as follows:
\[\partial_t+H_1t^{r_1}+H_2t^{r_2}+\cdots+H_mt^{r_m}+Xt^{-1},\]
where $r_1<r_2<\cdots<r_m<-1$, $H_i$ are diagonal matrices, $X$ is
an upper triangular matrix, and $[H_i,X]=0$. It is proved in
\cite{BV}, \S 6 that, possibly after a finite field extension
$E/F$, for every connection $(V,\nabla)$, there exists some
($E$-)basis $\mathbf{e}$ of $V\otimes_FE$, such that this
connection is in a canonical form with respect to this basis.

Now we begin to prove the lemma. Assume that $\dim V=n$. Let
$E=F(t^{1/d})$ be a finite extension of $F$ and an $E$-basis
$\mathbf{e}$ of $V\otimes_FE$ such that the connection $\nabla$
with respect to this basis is in the canonical form. Let
$\Lambda=\calO_E\mathbf{e}$, where $\calO_E$ is the integral
closure of $\calO$ in $E$. $\Lambda$ is a lattice of
$V\otimes_FE$. Since the connection is in the canonical form with
respect to $\mathbf{e}$, we have:
\begin{equation}\label{canonical} \nabla^{-1}(t^k\Lambda)\subset
  t^{k+1}\Lambda+\text{Sol},\quad\text{for any $k$},
\end{equation}
where $\text{Sol}\subset V\otimes_FE$ is the kernel of $\nabla$ (that
is, the solution space of $\nabla$). Note that
$\dim_{\C}(\text{Sol})\leq n$.

Set $M=\Lambda\cap V$ inside $V\otimes_FE$. Then $M$ is a lattice
in $V$. By \eqref{canonical},
$$[\nabla^{-1} (t^kM): t^kM]\le 0,$$
because the codimension of $t^{k+1}M$ in
$V\cap(t^{k+1}\Lambda+\text{Sol})$ is at most $n$.

Finally, we can prove the statement. Indeed, take any lattice $L$ and
choose $k$ such that $L\supset t^kM$. Since
$$\dim(L/t^kM)\ge\dim(\nabla^{-1}L/\nabla^{-1}(t^kM)),$$ the statement
follows.

\end{proof}

\begin{rmk}
Observe that the tangent spaces for the non-deformed affine
Springer fiber are never finite-dimensional (even for regular
semisimple elements in $\frakg(F)$). This is because the
non-deformed affine Springer fiber of a regular semisimple element
is highly non-reduced and has infinitely many ``nilpotent
directions''.
\end{rmk}

\section{Proof of Proposition \ref{action of affine Weyl group}}

Let $I_{\aff}=I_{\on{f}}\bigsqcup\{i_0\}$ be the set of vertices in
the affine Dynkin diagram for $G(F)$, with $i_0$ corresponding to the
affine vertex. Denote by $T_{s_i}$ to the simple reflection
corresponding to the vertex $i\in I_{\aff}$. It is enough to
construct, for each $i$, an involution $\sigma_i:H_*(X_A)\to H_*(X_A)$
such that the natural map $j:H_*(X_A)\to H_*(\Fl)$ satisfies
$j(\sigma_i(x))=T_{s_i}j(x)$.

Let $\Aut^0(D)$ be the group of automorphisms of $D=\spec\OO$.  It is
an extension of $\bbG_m$ by a pro-unipotent group $\Aut^+(D)$ (see,
e.g., \cite{FB} \S 6.2). The Lie algebra $\Der^0(D)$ of $\Aut^0(D)$ is
topologically spanned by $\{t^n\partial_t;n\geq 1\}$ and the Lie
algebra $\Der^+(D)$ of $\Aut^+(D)$ is topologically spanned by
$\{t^n\partial_t;n\geq 2\}$. $\Aut^0(D)$ acts on $G(F)$, and we can
form the semi-direct product $G(F)\rtimes\Aut^0(D)$. We have
\[\Lie (G(F)\rtimes\Aut^0(D))=\frakg(F)\oplus\Der^0(D) \ \ \
\mbox{ as vector spaces}.\]
Obviously, the action of $\Aut^0(D)$ on $G(F)$ leaves $G(\OO)$
invariant. Therefore, it acts on $\Gr$. We thus obtain an action
of $G(F)\rtimes\Aut^0(D)$ on $\Gr$. In a similar fashion,
$G(F)\rtimes\Aut^0(D)$ acts on all the affine (partial) flag
varieties of $G(F)$, as is seen from the following lemma.

A standard parahoric subgroup of $G(F)$ is a parahoric subgroup of
$G(F)$ that contains $I$. For $i\in I_{\aff}$, let $P_i$ be the
standard minimal parahoric subgroup corresponding to $i$.

\begin{lem}\label{invariant}The action of $\Aut^0(D)$ on $G(F)$ leaves
$I,P_i, i\in I_{\aff}$ invariant and therefore leaves all standard
parahoric subgroups of $G(F)$ invariant.
\end{lem}
\begin{proof}Write $G(\OO)=G^{(1)}(\OO)G$. The action of
$\Aut^0(D)$ on $G(\OO)$ leaves $G^{(1)}(\OO)$ invariant and
fixes $G$. Since, $I$ and $P_i, i\in I_{\on{f}}$ are pre-images of
subgroups of $G$ under the evaluation map $G(\OO)\to G$, they
are invariant under the action of $\Aut^0(D)$. It remains to show
that $P_{i_0}$ is also invariant under the action of $\Aut^0(D)$,
where $i_0$ is the affine vertex in the affine Dynkin diagram of
$\frakg$.

We have $\Lie P_{i_0}=\mathrm{Lie}I+t^{-1}\frakg_\theta$, where
$\frakg_\theta$ the the root space corresponding to the highest
root $\theta$. It is clear that
$[\Der^0(D),t^{-1}\frakg_\theta]\subset t^{-1}\C[[t]]\frakg_\theta
\subset \Lie P_{i_0}$. Therefore, the action of $\Aut^0(D)$ also
leaves $P_{i_0}$ invariant. Since the standard parahoric subgroups
are generated by some of the $P_i$'s, the lemma follows.
\end{proof}

Thus, elements in the Lie algebra $\Lie(G(F)\rtimes\Aut^0(D))$ act on
these affine (partial) flag varieties by vector fields. The zero sets
of these vector fields are nothing but our deformed affine Springer
fibers! The reason is the following. The group $G(F)$ acts on
$\Lie(G(F)\rtimes\Aut^0(D))$ via the adjoint representation.  Let us
denote this adjoint representation by $\widetilde{\Ad}$ to distinguish
it from the adjoint representation of $G(F)$ on $\frakg(F)$. Let
$$
(\wt{A},t^{-r}\partial_t) \in\frakg(F) \oplus \Der^0(D), \qquad r\leq
-1.
$$

We have
\begin{lem} For $g\in G(F)$,
$\widetilde{\Ad}_g((\wt{A},t^{-r}\partial_t))=(\Ad_g(\wt{A})-t^{-r}
(\partial_tg)g^{-1},t^{-r}\partial_t)$.
\end{lem}

\begin{proof} Let $B\in\frakg(F)$. We have
\begin{align*}
[\widetilde{\Ad}_g((\wt{A},t^{-r}\partial_t)),B] &=
\widetilde{\Ad}_g[(\wt{A},t^{-r}\partial_t),\Ad_{g^{-1}}(B)]\\
&=[\Ad_g(\wt{A}),B] + \Ad_g[t^{-r}\partial_t,\Ad_{g^{-1}}(B)]\\
&=[\Ad_g(\wt{A}),B]+\Ad_g([t^{-r}\partial_t(g^{-1})g,\Ad_{g^{-1}}(B)]
+\Ad_{g^{-1}}(t^{-r}\partial_tB))\\
&=[(\Ad_g(\wt{A})-t^{-r}(\partial_tg)g^{-1},t^{-r}\partial_t),B].
\end{align*}
Since $\frakg$ is semisimple, this identity implies the desired
formula.
\end{proof}

Therefore, if $\wt{A}, r$ are as in the assumption of Theorem
\ref{regular point}, we obtain that the reduced algebraic variety
$X^{\on{red}}_A\subset\Fl$ underlying $X_A$ is the zero set of the
vector field on $\Fl$ obtained by the action of
$(\wt{A},t^{-r}\partial_t)\in\Lie(G(F)\rtimes\Aut^0(D))$.
Likewise, $Y^{\on{red}}_A\subset\Gr$ is the zero set of the
corresponding vector field on $\Gr$. Let $\Fl_i=G(F)/P_i$, and
$\pi_i:\Fl\to\Fl_i$ be the projection. This is a
$\bbP^1$-fibration. We will also define $X^{\on{red}}_{A,i}$ to be
the zero set of the corresponding vector field on $\Fl_i$. It is
clear that the projection $\pi_i:\Fl\to\Fl_i$ restricts to
$\pi_i:X^{\on{red}}_A\to X^{\on{red}}_{A,i}$.

Now, under the assumptions of Theorem \ref{regular point}, $r\leq -2$,
and $\wt{A}\in\Lie I^0$, where $I^0=[I,I]$ is the pro-unipotent
radical of $I$. Therefore,
\[(\wt{A},t^{-r}\partial_t)\in\Lie I^0 \oplus
\Der^+(D)=\Lie(I^0\rtimes \Aut^+(D)).\] Since $I^0\rtimes\Aut^+(D)$ is
pro-unipotent, the vector field on $\Fl$ (resp., on $\Fl_i$) gives
rise to an action of $\bbG_a$ on $\Fl$ (resp., on
$\Fl_i$). Furthermore, the projection $\pi_i:\Fl\to\Fl_i$ is
$\bbG_a$-equivariant. Now $X^{\on{red}}_{A,i}$ is just the fixed point
set of this $\bbG_a$ action on $\Fl_i$. Therefore, there is a
fiberwise $\bbG_a$-action on $\pi_i^{-1}(X^{\on{red}}_{A,i})$, which
is a $\bbP^1$-fibration over $X^{\on{red}}_{A,i}$, and
$X^{\on{red}}_A$ is just the fixed point set. Now the construction of
\cite{KL1} \S 2 gives us the desired involution $\sigma_i:H_*(X_A)\to
H_*(X_A)$.

This completes the proof of Proposition \ref{action of affine Weyl
  group} and hence of Theorem \ref{regular point}. Therefore Theorem
  \ref{main} is now proved.

\section{The action of the affine Weyl group on $H_*(X_A)$}

We continue to assume that $G$ is a connected simply-connected
semisimple complex algebraic group. Let $A$ be a regular semisimple
nil-element in $\frakg(F)$, i.e., $(\on{ad}A)^r\to 0$ if $r\to\infty$,
as defined in \cite{KL} \S 2. According to \emph{loc. cit.}, this is
equivalent to the property that $A$ is conjugate to an element of
$\g(\OO)$ whose reduction modulo $t$ is a nilpotent element of
$\g$. Let $\on{Sp}_A$ be the \emph{non-deformed} affine Springer fiber
of $A$ in $\Fl$. In \cite{Lu} \S 5, Lusztig constructed an action of
$W_{\aff}$ on $H_*(\on{Sp}_A)$. We show in this section that a similar
construction can be applied to obtain an action of $W_{\aff}$ on the
homology of the \emph{deformed} affine Springer fibers.

Let $(\wt{A},t^{-r}\partial_t) \in \Lie I^0 \oplus \Der^+(D)$. We
will prove that the homology $H_*(X_A)$ itself admits an action of
the affine Weyl group, where the simple reflection corresponding
to $i$ will act on $H_*(X_A)$ by $\sigma_i$ constructed above. The
only new result here is Proposition \ref{Groth alteration}, the
counterpart of which for the usual affine Springer fiber is proved
in \cite{Lu} \S 5.4.

For every $J\subsetneqq I_{\aff}$, let $P_J$ be the standard
parahoric subgroup of $G(F)$, generated by $P_i$, $i\in J$. This
is a pro-algebraic group. Let $P_J^u$ be its pro-unipotent
radical, so that $G_J:=P_J/P_J^u$ is a reductive group. Let
$\frakg_J=\Lie G_J$. For example, if $J=I_{\on{f}}$ is the set of
vertices in the finite Dynkin diagram, then
$P_{I_{\on{f}}}=G(\OO)$, $P_{I_{\on{f}}}^u=G^{(1)}(\OO)$ and
$G_{I_{\on{f}}}=G$. The construction is based on the following

\begin{lem}\label{A} Let $J\subsetneqq I_{\aff}$. Then for any $g\in
  P_J$ and $r\leq -2$ we have $t^{-r}(\partial_tg)g^{-1}\in\Lie
  P_J^u$.
\end{lem}

\begin{proof} It is enough to show that in
$\Lie(G(F)\rtimes\Aut^0(D))$, $[t^{-r}\partial_t,\Lie P_J]\subset
\Lie P_J^u$. It is easy to see that
\[t^2\frakg(\OO)\subset\Lie P_J^u\subset\Lie P_J\subset t^{-1}
\frakg(\OO)\]
for any $J\subsetneqq I_{\aff}$. First we assume that $i_0\not\in
J$. In this case, $\Lie P_J\subset\frakg(\OO)$ and therefore
$[t^{-r}\partial_t,\Lie P_J]\subset t^2\frakg(\OO)\subset\Lie
P_J^u$. The lemma holds. Next, we assume that $J=\{i_0\}\cup J'$,
with $J'\subsetneqq I_{\on{f}}$. Then
\[\Lie P_J=\Lie P_J\cap\frakg(\OO)+\sum_{\beta\in\Delta_{J'}^+\cup\{0\}}
t^{-1}\frakg_{\theta-\beta},\]
where $\Delta_{J'}^+$ is the set of positive roots for $G_{J'}$.
Clearly, $[t^{-r}\partial_t,\Lie P_J\cap\frakg(\OO)]\in\Lie P_J^u$. In
addition, $[t^{-r}\partial_t,t^{-1}\frakg_{\theta-\beta}]=t^{-r-2}
\frakg_{\theta-\beta}$, which belongs to $\Lie P_J$. But since
$t^{r+2}\frakg_{\beta-\theta}\nsubseteq\Lie P_J$, $t^{-r-2}
\frakg_{\theta-\beta}$ indeed belongs to $\Lie P_J^u$. The lemma
follows.
\end{proof}

This lemma may also be reformulated as follows: the induced action
of $\Aut^+(D)$ on $G_J=P_J/P_J^u$ is trivial.

Let
\[X_{A,J}=\{g\in G(F)|\Ad_{g^{-1}}(\wt{A})-t^{-r}d\log(g^{-1})\in
\Lie P_J\}/P_J\subset G(F)/P_J\]
be the deformed Springer fiber in $G(F)/P_J$. By Lemma \ref{A},
this is well-defined. For example, if $J=I_{\on{f}}$, then
$X_{A,J}=Y_A$. The natural projection $\pi_J:\Fl\to G(F)/P_J$
restricts to a map $\pi_J:X_A\to X_{A,J}$.

Let $\wt{\frakg}_J\stackrel{p_J}{\to}\frakg_J$ be the
Grothendieck alteration of $\frakg_J$, which classifies pairs
consisting of a Borel subalgebra of $\frakg_J$ and an element
contained in this subalgebra.

\begin{prop}\label{Groth alteration} There is a natural Cartesian
  diagram
\[\begin{CD}
X_A@>>>[\wt{\frakg}_J/G_J]\\
@V\pi_J VV@VVp_JV\\
X_{A,J}@>\varphi_J>>[\frakg_J/G_J]
\end{CD}\]
\end{prop}

\begin{proof} We first construct the morphisms $X_{A,J}\to
  \frakg_J/G_J$. Let $\wt{X}_{A,J}$ be the preimage of $X_{A,J}$ under
the projection $G(F)/P_J^u\to G(F)/P_J$. We have
\[\wt{X}_{A,J}=\{g\in G(F)|\Ad_{g^{-1}}(\wt{A})-t^{-r}
\partial_t(g^{-1})g\in\Lie P_J\}/P^u_J\]
By Lemma \ref{A}, the map
\[gP^u_J\mapsto \Ad_{g^{-1}}(\wt{A})-t^{-r}\partial_t(g^{-1})g \ \mod
\Lie P_J^u\]
is a well define $G_J$-equivariant map $\wt{X}_{A,J}\to
\frakg_J$. This gives the desired map
$\varphi_J:X_{A,J}\to\frakg_J/G_J$.

Let $\wt{X}_A:=X_A\times_{X_{A,J}}\wt{X}_{A,J}$, so that
$\wt{X}_{A}$ classifies the pairs $(gI,g'P_J^u), g,g'\in G(F)$
such that $gP_J=g'P_J$ and
\[\Ad_{g^{-1}}(\wt{A})-t^{-r}\partial_t(g^{-1})g\in\Lie I, \ \
\Ad_{g'^{-1}}(\wt{A})-t^{-r}\partial_t(g'^{-1})g'\in\Lie P_J.\]
On the other hand,
$\wt{\wt{X}}_A:=\wt{X}_{A,J}\times_{\frakg_J}\wt{\frakg}_J$
classifies pairs $(gI,g'P_J^u), g\in P_J, g'\in G(F)$ such that
\[\Ad_{g'^{-1}}(\wt{A})-t^{-r}\partial_t(g'^{-1})g'\in
\Ad_{g}(\Lie I)\subset\Lie P_J.\]

Let $(gI,g'P_J^u)\in\wt{\wt{X}}_A$. We find that
\[\Ad_{(g'g)^{-1}}(\wt{A})-t^{-r}\partial_t((g'g)^{-1})(g'g)=
\Ad_{g^{-1}}(\Ad_{g'^{-1}}(\wt{A})-t^{-r}\partial_t(g'^{-1})g')
-t^{-r}\partial_t(g^{-1})g\]
is in $\Lie I$. This is because
$\Ad_{g'^{-1}}(\wt{A})-t^{-r}\partial_t(g'^{-1})g'\in\Ad_g(\Lie
I)$ and $t^{-r}\partial_t(g^{-1})g\in\Lie P_J^u\subset\Lie I$ by
Lemma \ref{A}. Therefore, $(g'gI,g'P_J^u)\in\wt{X}_A$.
Conversely, if $(gI,g'P_J^u)\in\wt{X}_A$, then
$(g'^{-1}gI,g'P_J^u)\in\wt{\wt{X}}_A$. Therefore, there is a
$G_J$-equivariant isomorphism $\wt{X}_A\to\wt{\wt{X}}_A$
sending $(gI,g'P_J^u)\to (g'^{-1}gI,g'P_J^u)$.

Thus, we obtain a Cartesian diagram
\[\begin{CD}
X_A\times_{X_{A,J}}\wt{X}_{A,J}@>>>\wt{\frakg}_J\\
@VVV@VVV\\
\wt{X}_{A,J}@>>>\frakg_J
\end{CD}\]
where all morphisms are $G_J$-equivariant. The proposition follows
by taking the $G_J$-quotients.
\end{proof}

Let $\underline\bbC$ be the constant sheaf on $X_A$. Then
$(\pi_J)_*\underline\bbC=\varphi_J^*(p_J)_*\underline\bbC$. By the
Springer theory for finite Weyl group, we obtain an action of $W_J$
(the finite Weyl group of $G_J$) on $(\pi_J)_*\underline\bbC$.
Therefore, we obtain a representation of $W_J$ on $H_*(X_A)$.
Following the argument of \cite{Lu} \S 5.5, we obtain that these
representations for all $J\subsetneqq I_{\aff}$ give rise to a
representation of $W_{\aff}$ on $H_*(X_A)$.

\end{document}